\documentclass[11pt]{amsart}
\usepackage{amsmath,amsthm}
\usepackage{cpma}
\usepackage{enumerate}

\newtheorem{thm}{Theorem}
\newtheorem{prop}[thm]{Proposition}

\newtheorem{lem}[thm]{Lemma}

\newcommand{\ld}{\log\det}

\newcommand{\ub}{\underline{u}{}}
\newcommand{\ccnorm}[1]{\norm{#1}_{\cC^2(\bar{\fO})}}
\newcommand{\llnorm}[1]{\norm{#1}_{L^2(\fO)}}
\newcommand{\lnorm}[1]{\norm{#1}_{L^1(\fO)}}

\DeclareMathOperator{\dist}{dist}
\DeclareMathOperator{\re}{Re}

\begin{document}
\title{Energy Functionals for the Parabolic Monge-Amp\`{e}re Equation}

\author{Zuoliang Hou}
\address{Mathematics Department, Columbia University, New York, NY 10027}
\email{hou@math.columbia.edu}
\author{Qi Li}
\address{Mathematics Department, Columbia University, New York, NY 10027}
\email{liqi@math.columbia.edu}
\date{\today}
\maketitle

\section{Introduction}

Because of its close connection with the \KH-Ricci flow, the parabolic
complex \MA{} equation on complex manifolds has been studied by many
authors.  See, for instance, \cite{Cao1985, ChenTian2002,
PhongSturm2006}.
On the other hand, theories for complex \MA{} equation on both
bounded domains and complex manifolds were developed in \cite{BedfordTaylor1976,
Yau1978, CKNS1985, Kolodziej1998}. 
In this paper, we are going to study the parabolic complex
\MA{} equation over a bounded domain. 

\

Let $\fO \subset \CC^n$ be a bounded domain with smooth boundary $\di\fO$.
Denote $\cQ_T=\fO \times
(0,T)$ with $T>0$, $B=\fO \times \left\{ 0 \right\} $, $\fG=\di\fO \times
\left\{ 0 \right\}$ and $\fS_T=\di\fO \times (0, T)$. Let $\di_p
\cQ_T$ be the parabolic boundary of $\cQ_T$, i.e. $\di_p \cQ_T = B
\cup \fG \cup \fS_T$. Consider the following boundary value problem:
\begin{equation}
  \left\{
  \begin{aligned}
    & \dod{u}{t} - \ld\big(u_{\fa\bar{\fb}}\big) = f(t, z, u) 
    && \text{ in } \cQ_T, \\ 
    & u= \vff           & & \text{ on } \di_p \cQ_T. 
  \end{aligned}
  \right.
  \label{eq:1}
\end{equation}
where $f \in \cC^{\infty}(\RR \times \bar{\fO} \times \RR)$ and $\vff
\in \cC^{\infty}(\di_p\cQ_T)$.
We will always assume that
\begin{equation}
  \dod{f}{u} \leq 0.
  \label{eq:2}
\end{equation}
Then we will prove that 

\begin{thm}
  Suppose there exists a spatial plurisubharmonic (psh) function  $\ub \in \cC^2(\bar{\cQ}_T)$
  such that
  \begin{equation}
  \left.
  \begin{aligned}
    & {\ub\,}_t - \ld \big({\ub\,}_{\fa\bar{\fb}}\big) \leq f(t, z, \ub)
    \qquad \qquad \text{ in } \cQ_T, \\
    & \ub \leq \vff \quad \text{on }\; B  \qquad \text{and} \qquad 
    \ub = \vff \quad \text{on }\; \fS_T \cap \fG.  
  \end{aligned}
  \right.
  \label{eq:3}
  \end{equation}
  Then there exists a spatial psh solution $u \in
  \cC^{\infty}(\bar{\cQ}_T)$ of (\ref{eq:1}) with $u \geq \ub$ if
  following compatibility condition is satisfied:
  $\forall\,z \in\di\fO$,
  \begin{equation}
    \begin{split}
     \vff_t - \ld\big( \vff_{\fa\bar{\fb}} \big) &= f(0, z, \vff(z)),\\
     \vff_{tt}  - \big( \ld (\vff_{\fa\bar{\fb}}) \big)_t &=
    f_t(0,z,\vff(z)) + f_u(0,z,\vff(z))\vff_t.
      \end{split}
    \label{eq:4}
  \end{equation}
\end{thm}

Motivated by the energy functionals in the study of the \KH-Ricci
flow, we introduce certain energy functionals to the complex \MA{}
equation over a bounded domain.  
Given $\vff \in \cC^{\infty}(\di\fO)$, denote
\begin{equation}
  \cP(\fO, \vff) = \left\{ u \in \cC^2(\bar{\fO}) \,\mid\, 
  u \text{ is psh, and } u = \vff \text{ on } \di\fO \right\},
  \label{eq:5}
\end{equation}
then define the $F^0$ functional by following variation formula:
\begin{equation}
  \fd F^0(u) = \int_{\fO} \fd u \det \big( u_{\fa\bar{\fb}} \big).
  \label{eq:6}
\end{equation}
We shall show that the $F^0$ functional is well-defined. 
Using this $F^0$ functional and following the ideas of
\cite{PhongSturm2006}, we prove that

\begin{thm}
  Assume that both $\vff$ and $f$ are independent of $t$, and 
  \begin{equation}
    f_{u} \leq 0  \qquad \text{and} \qquad f_{uu} \leq 0.
    \label{eq:7}
  \end{equation}
  Then the solution $u$ of (\ref{eq:1}) exists for $T = + \infty$, and as
  $t$ approaches $+\infty$, $u(\cdot,t)$ approaches the unique
  solution of the Dirichlet problem
  \begin{equation}
    \left\{
    \begin{aligned}
      & \det \big(v_{\fa\bar{\fb}}\big) = e^{-f(z, v)}
      && \text{ in } \cQ_T, \\ 
      & v= \vff           & & \text{ on } \di_p \cQ_T, 
    \end{aligned}
    \right.
    \label{eq:8}
  \end{equation}
  in $\cC^{1,\fa}(\bar{\fO})$ for any $0< \fa <1$.
  \label{thm:2}
\end{thm}

\noindent
\emph{Remark}: 
Similar energy functionals have been studied in \cite{Bakelman1983,
Tso1990, Wang1994, TrudingerWang1997, TrudingerWang1998} for the real
\MA{} equation and the real Hessian equation with homogeneous boundary
condition $\vff=0$, and the convergence for the solution of the real
Hessian equation was also proved in \cite{TrudingerWang1998}. Our
construction of the energy functionals and the proof of the convergence 
also work for these cases, and thus we also obtain an independent
proof of these results. 
Li~\cite{LiSY2004} and Blocki~\cite{Blocki2005} 
studied the Dirichlet problems for the complex $k$-Hessian equations
over bounded complex domains. Similar energy functional can also be
constructed for the parabolic complex $k$-Hessian equations and be
used for the proof of the convergence.

\section{A priori $\cC^2$ estimate}

By the work of Krylov~\cite{Krylov1983}, Evans~\cite{Evans1982}, Caffarelli
etc.~\cite{CKNS1985} and Guan~\cite{Guan1998}, it is well known that in
order to prove the existence
and smoothness of (\ref{eq:1}), we only need to establish the a priori
$\cC^{2,1}(\bar{\cQ}_T)$\footnote{$\cC^{m,n}(\cQ_T)$ means $m$ times
and $n$ times  differentiable in space direction and time
direction respectively, same for $\cC^{m,n}$-norm.} 
estimate, i.e.  for solution $u \in \cC^{4,1}(\bar{\cQ}_T)$ 
of (\ref{eq:1}) with
\begin{equation}
  u=\ub \quad \text{ on } \quad \fS_T\cup \fG  \qquad
  \text{and} \qquad 
  u \geq \ub \quad \text{ in } \quad \cQ_T,
  \label{eq:9}
\end{equation}
then
\begin{equation}
  \norm{u}_{\cC^{2,1}(\cQ_T)} \leq M_2,
  \label{eq:10}
\end{equation}
where  $M_2$ only depends on $\cQ_T, \ub, f $ and $\ccnorm{u(\cdot,0)}$.

\

\noindent
\emph{Proof of (\ref{eq:10})}. 
Since $u$ is spatial psh and $u \geq \ub$, so
\begin{equation*}
  \ub \leq u \leq \sup_{\fS_T} \ub 
\end{equation*}
i.e.
\begin{equation}
  \norm{u}_{\cC^0(\cQ_T)} \leq M_0.
  \label{eq:11}
\end{equation}

\

\noindent
\textbf{Step 1.} $|u_t| \leq  C_1$  in $\bar{\cQ}_T$.

Let $G=u_t (2M_0 - u)^{-1}$. If $G$ attains its minimum on
$\bar{\cQ}_T$ at the parabolic boundary, then $u_t  \geq -C_1$ where
$C_1$ depends on $M_0$ and ${\ub\,}_t$ on $\fS$. Otherwise, at the
point where $G$ attains the minimum, 
\begin{equation}
  \begin{split}
    G_t \leq 0 \quad &\text{i.e.} \quad u_{tt} + (2M_0 - u)^{-1} u_t^2
    \leq 0, \\
    G_{\fa} = 0 \quad &\text{i.e.} \quad u_{t\fa}+ (2M_0 - u)^{-1} u_t
    u_{\fa} = 0, \\
    G_{\bar{\fb}} = 0 \quad &\text{i.e.} \quad u_{t\bar{\fb}}+ 
    (2M_0 - u)^{-1} u_t u_{\bar{\fb}} = 0 ,
  \end{split}
  \label{eq:12}
\end{equation}
and the matrix $G_{\fa\bar{\fb}}$ is non-negative, i.e.
\begin{equation}
  u_{t\fa\bar{\fb}} + (2M_0 -u)^{-1} u_t u_{\fa\bar{\fb}} \geq 0.
  \label{eq:13}
\end{equation}
Hence
\begin{equation}
  0 \leq u^{\fa\bar{\fb}}\big( u_{t\fa\bar{\fb}} +
  (2M_0 -u)^{-1} u_t u_{\fa\bar{\fb}}  \big) = 
  u^{\fa\bar{\fb}} u_{t\fa\bar{\fb}} + n (2M_0 -u)^{-1} u_t,
  \label{eq:14}
\end{equation}
where $(u^{\fa\bar{\fb}})$ is the inverse matrix for
$(u_{\fa\bar{\fb}})$, i.e. 
\begin{equation*}
  u^{\fa\bar{\fb}} u_{\fg\bar{\fb}} = {\fd^{\fa}}_{\fg}. 
\end{equation*}
Differentiating (\ref{eq:1}) in $t$,  we get
\begin{equation}
  u_{tt} - u^{\fa\bar{\fb}} u_{t \fa\bar{\fb}} = f_t + f_u\, u_t,
  \label{eq:15}
\end{equation}
so
\begin{equation*}
  \begin{split}
    (2M_0 - u)^{-1} u_t^2 & \leq - u_{tt}  \\
    &= - u^{\fa\bar{\fb}}u_{t\fa\bar{\fb}} - f_t - f_u\, u_t\\
    &\leq n(2M_0 - u)^{-1} u_t - f_u\, u_t - f_t, 
  \end{split}
\end{equation*}
hence
\begin{equation*}
  u_t^2 - ( n - (2M_0 - u) f_u ) u_t + f_t (2M_0 - u) \leq 0 .
\end{equation*}
Therefore at point $p$, we get
\begin{equation}
  u_t \geq -C_1
  \label{eq:16}
\end{equation}
where $C_1$ depends on $M_0$ and $f$.

Similarly, by considering the function $u_t (2M_0 + u)^{-1}$ 
we can show that
\begin{equation}
  u_t \leq C_1.
  \label{eq:17}
\end{equation}

\

\noindent
\textbf{Step 2.} $|\nabla u|\leq M_1$

Extend $\ub|_{\fS}$ to a spatial harmonic function $h$, then
\begin{equation}
  \ub \leq u \leq h \quad \text{in} \quad \cQ_T \qquad \text{ and }
  \qquad \ub = u = h \quad \text{on} \quad \fS_T.
  \label{eq:18}
\end{equation}
So
\begin{equation}
  |\nabla u|_{\fS_T} \leq M_1.
  \label{eq:19}
\end{equation}

Let $L$ be the linear differential operator defined by
\begin{equation}
  L v = \dod{v}{t} - u^{\fa\bar{\fb}} v_{\fa\bar{\fb}} - f_u v.
  \label{eq:20}
\end{equation}
Then
\begin{equation}
  \begin{split}
    L(\nabla u + e^{\fl |z|^2}) &= L(\nabla u) + Le^{\fl |z|^2} \\
    & \leq \nabla f -  e^{\fl|z|^2}\big(\fl \sum u^{\fa\bar{\fa}} - f_u).
  \end{split}
  \label{eq:21}
\end{equation}
Noticed that and both $u$ and $\dot{u}$ are bounded and
\begin{equation*}
  \det \big( u_{\fa\bar{\fb}} \big) = e^{\dot{u} - f},
\end{equation*}
so
\begin{equation}
  0 < c_0 \leq \det\big( u_{\fa\bar{\fb}}\big) \leq c_1 ,
  \label{eq:22}
\end{equation}
where $c_0$ and $c_1$ depends on $M_0$ and $f$. Therefore
\begin{equation}
  \sum u^{\fa\bar{\fa}} \geq n c_1^{-1/n}.
  \label{eq:23}
\end{equation}
Hence after taking $\fl$ large enough, we can get
$$ L(\nabla u + e^{\fl |z|^2}) \leq 0, $$
thus
\begin{equation}
  |\nabla u| \leq \sup_{\di_p \cQ_T} |\nabla u| + C_2 \leq M_1.
  \label{eq:24}
\end{equation}

\

\noindent
\textbf{Step 3.} $|\nabla^2 u| \leq M_2$ on $\fS$.

At point $(p,t) \in \fS$, we choose coordinates $z_1, \cdots, z_n$
for $\fO$, 
such that at $z_1 = \cdots = z_n = 0 $ at $p$ and the positive $x_n$
axis is the interior normal direction of $\di \fO$ at $p$. We set
$ s_1 = y_1, s_2 = x_1, \cdots , s_{2n-1}=y_n, s_{2n}=x_n$ and 
$s'=(s_1, \cdots, s_{2n-1})$. We also assume that near $p$, 
$\di \fO$ is represented as a graph 
\begin{equation}
  x_n = \rho(s') = \frac{1}{2} \sum_{j,k < 2n } B_{jk}s_j s_k +
  O(|s'|^3).
  \label{eq:25}
\end{equation}
Since $(u-\ub)(s', \rho(s'), t)=0$, we have for $j, k < 2n$,
\begin{equation}
  (u-\ub)_{s_j s_k}(p, t) = - (u-\ub)_{x_n}(p, t) B_{jk},
  \label{eq:26}
\end{equation}
hence
\begin{equation}
  |u_{s_j s_k}(p,t)| \leq C_3,
  \label{eq:27}
\end{equation}
where $C_3$ depends on $\di\fO, \ub$ and $M_1$.

We will follow the construction of barrier function by
Guan~\cite{Guan1998} to estimate $|u_{x_n s_j}|$. For $\fd > 0$,
denote $\cQ_{\fd}(p,t) = \big(\fO \cap B_{\fd}(p)\big) \times (0,t)$.

\begin{lem} Define the function 
\begin{equation}
  d(z) = \dist(z, \di\fO)
  \label{eq:28}
\end{equation}
and
\begin{equation}
  v = (u-\ub) + a (h-\ub) - N d^2.
  \label{eq:29}
\end{equation}
Then for $N$ sufficiently large and $a, \fd$ sufficiently small,
\begin{equation}
  \left.
  \begin{aligned}
    L v &\geq \fe (1+ \sum u^{\fa\bar{\fa}}) && \text{ in } \cQ_{\fd}(p,t) \\
    v  &\geq 0  && \text{on } \di(B_{\fd}(p) \cap \fO) \times (0, t) \\
    v(z,0) & \geq c_3 |z| && \text{for } z \in B_{\fd}(p) \cap \fO  
  \end{aligned}
  \right.
  \label{eq:30}
\end{equation}
where $\fe$ depends on the uniform lower bound of he eigenvalues of
$\{{\ub\,}_{\fa\bar{\fb}}\}$.
\label{thm:3}
\end{lem}
\begin{proof}
  See the proof of Lemma~2.1 in \cite{Guan1998}.
\end{proof}

For $j< 2n$, consider the operator
$$ T_j = \dod{\,}{s_j} + \rho_{s_j} \dod{}{x_n}. $$
Then
\begin{equation}
  \left.
  \begin{aligned}
    T_j(u-\ub)  &=0 & & \text{on } \big(\di\fO \cap B_{\fd}(p)\big)
    \times (0, t) \\
    |T_j(u-\ub)| & \leq M_1 & & \text{on } \big(\fO \cap \di B_{\fd}(p)\big)
    \times (0, t) \\
    |T_j(u-\ub)(z, 0) | & \leq C_4 |z| & & \text{for } z \in B_{\fd}(p)  
  \end{aligned}
  \right.
  \label{eq:31}
\end{equation}
So by Lemma~\ref{thm:3} we may choose $C_5$ independent of $u$, and
$A >> B >> 1 $ so that
\begin{equation}
  \left.
  \begin{aligned}
    L\big(Av + B|z|^2 - C_5 (u_{y_n} -{\ub\,}_{y_n})^2 \pm T_j(u-\ub)
    \big) &\geq 0  && \text{in } \cQ_{\fd}(p, t),\\
    Av + B|z|^2 - C_5 (u_{y_n} -{\ub\,}_{y_n})^2 \pm T_j(u-\ub)
     &\geq 0  && \text{on } \di_p \cQ_{\fd}(p, t).
  \end{aligned}
  \right.
  \label{eq:32}
\end{equation}
Hence by the comparison principle, 
\begin{equation*}
    Av + B|z|^2 - C_5 (u_{y_n} -{\ub\,}_{y_n})^2 \pm T_j(u-\ub)
     \geq 0 \qquad \text{in } \cQ_{\fd}(p, t),
\end{equation*}
and at $(p,t)$
\begin{equation}
  |u_{x_n y_j}| \leq M_2.
  \label{eq:33}
\end{equation}

To estimate $|u_{x_n x_n}|$, we will follow the simplification in
\cite{Trudinger1995}. For $(p,t) \in \fS$, define
$$ \fl(p,t) = \min \{ u_{\xi\bar{\xi}} \,\mid\, \text{ complex vector
} \xi \in T_{p} \di \fO, \text{ and } |\xi| = 1 \} $$

\noindent
\textbf{Claim} $ \fl(p,t) \geq c_4 > 0 $ where $c_4$ is independent of
$u$.

Let us assume that $\fl(p,t)$ attains the minimum at $(z_0, t_0)$ with
$\xi \in T_{z_o} \di\fO$. We may assume that
$$ \fl(z_0,t_0) < \frac{1}{2}\, {\ub\,}_{\xi\bar{\xi}}(z_0,t_0). $$
Take a unitary frame $e_1, \cdots, e_n$ around $z_0$, such that
$e_1(z_0) = \xi$, and $\re e_n = \fg $ is the interior normal of $\di\fO$
along  $\di \fO$. Let $r$ be the function which defines $\fO$, then 
\begin{equation*}
  (u-\ub\,)_{1\bar{1}}(z,t) = - r_{1\bar{1}}(z)  (u-\ub\,)_{\fg}(z,t)
  \qquad  z \in \di\fO
\end{equation*}
Since $u_{1\bar{1}}(z_0, t_0) < {\ub\,}_{1\bar{1}}(z_0,t_0) /2 $, so 
$$ - r_{1\bar{1}}(z_0)  (u-\ub\,)_{\fg}(z_0,t_0) \leq - \frac{1}{2}\,
    {\ub\,}_{1\bar{1}}(z_0,t_0).  $$
Hence
$$  r_{1\bar{1}}(z_0)  (u-\ub\,)_{\fg}(z_0,t) \geq  \frac{1}{2}\,
    {\ub\,}_{1\bar{1}}(z_0,t) \geq c_5 > 0 . $$
Since both $\nabla u$ and $\nabla \ub$ are bounded, we get
$$  r_{1\bar{1}}(z_0) \geq c_6 > 0, $$
and for $\fd$ sufficiently small ( depends on $r_{1\bar{1}}$ )  and
$ z \in B_{\fd}(z_0)\cap \fO $, 
$$ r_{1\bar{1}}(z) \geq \frac{c_6}{2}. $$
So by $u_{1\bar{1}}(z,t) \geq u_{1\bar{1}}(z_0,t_0) $, we get
$$ \ub\,_{1\bar{1}}(z,t) - r_{1\bar{1}}(z)(u-\ub\,)_{\fg}(z,t)  \geq
\ub\,_{1\bar{1}}(z_0,t_0) - r_{1\bar{1}}(z_0)(u-\ub\,)_{\fg}(z_0,t_0). $$
Hence if we let 
$$ \Psi(z,t) = \frac{1}{r_{1\bar{1}}(z)} 
\big( r_{1\bar{1}}(z_0)(u-\ub\,)_{\fg}(z_0,t_0) +
\ub\,_{1\bar{1}}(z,t) - \ub\,_{1\bar{1}}(z_0,t_0)\big)$$
then
\begin{equation*}
  \left.
  \begin{aligned}
    (u-\ub\,)_{\fg}(z,t) &\leq \Psi(z,t) && \text{ on } 
    \big(\di\fO \cap B_{\fd}(z_0)\big)\times (0, T)\\
    (u-\ub\,)_{\fg}(z_0, t_0) &= \Psi(z_0, t_0).
  \end{aligned}
  \right.
\end{equation*}

Now take the coordinate system $z_1,\cdots, z_n$ as before. Then 
\begin{equation}
  \left.
  \begin{aligned}
    (u-\ub\,)_{x_n}(z,t) &\leq \frac{1}{\fg_n(z)} \Psi(z,t) && \text{ on } 
    \big(\di\fO \cap B_{\fd}(z_0)\big)\times (0, T)\\
    (u-\ub\,)_{x_n}(z_0, t_0) &= \frac{1}{\fg_n(z_0)}\Psi(z_0, t_0).
  \end{aligned}
  \right.
  \label{eq:34}
\end{equation}
where $\fg_n$ depends on $\di\fO$. After taking $C_6$ independent of
$u$ and $A >> B >> 1 $, we get
\begin{equation*}
  \left.
  \begin{aligned}
    L\big(Av + B|z|^2 - C_6 (u_{y_n} -{\ub\,}_{y_n})^2 
     + \frac{\Psi(z,t)}{\fg_n(z)} - T_j(u-\ub)
    \big) &\geq 0  && \text{in } \cQ_{\fd}(p, t),\\
    Av + B|z|^2 - C_6 (u_{y_n} -{\ub\,}_{y_n})^2 
     + \frac{\Psi(z,t)}{\fg_n(z)} - T_j(u-\ub)
     &\geq 0  && \text{on } \di_p \cQ_{\fd}(p, t).
  \end{aligned}
  \right.
\end{equation*}
So
$$    Av + B|z|^2 - C_6 (u_{y_n} -{\ub\,}_{y_n})^2 
     + \frac{\Psi(z,t)}{\fg_n(z)} - T_j(u-\ub)
     \geq 0 \qquad \text{in }  \cQ_{\fd}(p, t), $$
and
$$  |u_{x_n x_n}(z_0,t_0)| \leq C_7. $$
Therefore at $(z_0, t_0)$, ${u_{\fa\bar{\fb}}}$ is uniformly bounded,
hence 
$$  u_{1\bar{1}}(z_0, t_0) \geq c_4 $$
with $c_4$ independent of $u$. Finally, from the equation 
$$ \det u_{\fa \bar{\fb}} = e^{\dot{u} - f} $$
we get 
$$ |u_{x_n x_n}| \leq M_2. $$

\

\noindent
\textbf{Step 4.} $|\nabla^2 u| \leq M_2$ in $\cQ$.

By the concavity of $\ld$, we have
\begin{equation*}
    L(\nabla^2 u + e^{\fl |z|^2}) \leq O(1) - e^{\fl|z|^2}\big( \fl \sum
    u^{\fa\bar{\fa}} - f_u \big)
\end{equation*}
So for $\fl$ large enough, 
$$ L(\nabla^2 u + e^{\fl |z|^2}) \leq 0, $$
and
\begin{equation}
  \sup |\nabla^2 u| \leq \sup_{\di_p \cQ_T} |\nabla^2 u| + C_8
  \label{eq:35}
\end{equation}
with $C_8$ depends on $M_0$, $\fO$ and $f$. 

\qed 

\section{The Functionals $I, J$ and $F^0$}

Let us recall the definition of $\cP(\fO, \vff)$ in (\ref{eq:5}),
$$ \cP(\fO, \vff) = \left\{ u \in \cC^2(\bar{\fO} \,\mid\, 
  u \text{ is psh, and } u = \vff \text{ on } \di\fO \right\}. $$
Fixing $v \in \cP$, for $u\in \cP$, define
\begin{equation}
  I_v(u) = - \int_{\fO}(u-v) (\ii \ddb u)^n. 
  \label{eq:36}
\end{equation}

\begin{prop}
  There is a unique and well defined functional $J_v$ on
  $\cP(\fO,\vff)$, such that
  \begin{equation}
    \fd J_v(u) = - \int_{\fO} \fd u \big( (\ii \ddb u)^n  - (\ii\ddb
    v)^n \big),
    \label{eq:37}
  \end{equation}
  and $J_v(v)=0$.
  \label{thm:4}
\end{prop}

\begin{proof}
  Notice that $\cP$ is connected, so we can connect $v$ to $u\in \cP$
  by a path $u_t, 0 \leq t \leq 1$ such that $u_0=v$ and $u_1=u$.
  Define
  \begin{equation}
    J_v(u) = - \int_0^1 \int_{\fO} \dod{u_t}{t} \big( (\ii \ddb u_t)^n 
    - (\ii \ddb v)^n \big) \,dt.
    \label{eq:38}
  \end{equation}
  We need to show that the integral in (\ref{eq:38}) is independent of
  the choice of path $u_t$.  Let $\fd u_t = w_t $ be a variation of
  the path. Then
  $$ w_1 = w_0 = 0 \qquad \text{ and } \qquad w_t=0 \quad \text{on }
  \di\fO, $$
  and 
  \begin{equation*}
    \begin{split}
      & \qquad  \fd \int_0^1 \int_{\fO} \dot{u} \, \big( (\ii \ddb
      u)^n - (\ii \ddb v)^n  \big) \, dt\\
      &= \int_0^1 \int_{\fO} \Bl \dot{w} \bl (\ii \ddb
      u)^n - (\ii \ddb v)^n \br +  \dot{u}\, n  \ii \ddb w (\ii\ddb u)^{n-1}
      \Br \,dt ,
    \end{split}
  \end{equation*}
  Since $w_0=w_1=0$, an integration by part with respect to $t$ gives
  \begin{equation*}
    \begin{split}
      & \qquad  \int_0^1 \int_{\fO} \dot{w} \bl (\ii \ddb
      u)^n - (\ii \ddb v)^n \br \, dt \\ 
      &= - \int_0^1 \int_{\fO}  w \frac{d}{dt}(\ii \ddb u)^n \, dt
      = - \int_0^1 \int_{\fO} \ii n  w  \ddb \dot{u} (\ii\ddb u)^{n-1}\,  dt.
    \end{split}
  \end{equation*}
  Notice that both $w$ and $\dot{u}$ vanish on $\di\fO$, so an
  integration by part with respect to $z$ gives
  \begin{equation*}
    \begin{split}
     \int_{\fO} \ii n  w  \ddb \dot{u} (\ii\ddb u)^{n-1} &= -
     \int_{\fO} \ii n  \di w \wedge \dib \dot{u} (\ii\ddb u)^{n-1} \\
     &= \int_{\fO} \ii n  \dot{u} \ddb w (\ii\ddb u)^{n-1} .
   \end{split}
  \end{equation*}
  So
  \begin{equation}
      \fd \int_0^1 \int_{\fO} \dot{u} \, \big( (\ii \ddb u)^n - 
      (\ii \ddb v)^n  \big) \, dt = 0, 
      \label{eq:39}
  \end{equation}
  and the functional $J$ is well defined.
\end{proof}

Using the $J$ functional, we can define the $F^0$ functional as
\begin{equation}
  F_v^0(u) = J_v(u) - \int_\fO u (\ii\ddb v)^n.
  \label{eq:40}
\end{equation}
Then by Proposition~\ref{thm:4}, we have
\begin{equation}
  \fd F_v^0(u) = - \int_{\fO} \fd u (\ii\ddb u)^n.
  \label{eq:41}
\end{equation}

\begin{prop} The basic properties of $I, J$ and $F^0$ are following:
  \begin{enumerate}
    \item For any $u \in \cP(\fO, \vff)$, $I_v(u) \geq J_v(u) \geq 0.$
    \item $F^0$ is convex on $\cP(\fO,\vff)$, i.e. $\forall\, u_0,
      u_1 \in \cP$, 
      \begin{equation}
        F^0\bl \frac{u_0+u_1}{2} \br \leq \frac{F^0(u_0) +
        F^0(u_1)}{2}.
        \label{eq:42}
      \end{equation}
    \item $F^0$ satisfies the \emph{cocycle condition}, i.e.
      $\forall\, u_1, u_2, u_3 \in \cP(\fO, \vff)$, 
      \begin{equation}
        F_{u_1}^0(u_2) + F_{u_2}^0(u_3) = F_{u_1}^0(u_3).
      \end{equation}
  \end{enumerate}
\end{prop}
\begin{proof}
  Let $w = (u-v)$ and $u_t = v+ tw = (1-t)v + tu$, then
  \begin{equation}
    \begin{split}
      I_v(u) &= - \int_{\fO} w \bl (\ii\ddb u)^n - (\ii \ddb v)^n \br \\
      &= - \int_{\fO} w \bl \int_0^1 \frac{d}{dt} (\ii \ddb u_t)^n \,dt \br \\
      &= - \int_0^1 \int_{\fO} \ii\, n  w  \ddb w (\ii \ddb u_t)^{n-1} \\
      &= \int_0^1 \int_{\fO}  \ii\, n  \di w \wedge \dib w \wedge
      (\ii \ddb u_t)^{n-1} 
       \geq 0,
    \end{split}
    \label{eq:43}
  \end{equation}
  and 
  \begin{equation}
    \begin{split}
      J_v(u) &= - \int_0^1 \int_{\fO} w \bl (\ii\ddb u_t)^n - (\ii \ddb v)^n 
        \br \, dt \\ 
      &=- \int_0^1 \int_{\fO} w \bl \int_0^t \frac{d}{ds} 
      (\ii \ddb u_s)^n \, ds \br \, dt  \\ 
      &=- \int_0^1 \int_{\fO} \int_0^t \ii\, n w  \ddb w (\ii \ddb
      u_s)^{n-1} \, ds\, dt \\
      &= \int_0^1 \int_{\fO} (1-s) \ii\, n \di w \wedge  \dib w
      \wedge (\ii \ddb u_s)^{n-1} \,ds \geq 0. 
    \end{split}
    \label{eq:44}
  \end{equation}
  Compare (\ref{eq:43}) and (\ref{eq:44}), it is easy to see that 
  $$ I_v(u) \geq J_v(u) \geq 0 . $$
  To prove (\ref{eq:42}), let $u_t = (1-t) u_0 + t u_1$, then
  \begin{equation*}
    \begin{split}
      F^0(u_{1/2}) - F^0(u_0) &=- \int_0^{\frac{1}{2}} \int_{\fO} 
      ( u_1 - u_0 )\,(\ii \ddb  u_t )^n\,dt, \\
      F^0(u_1) - F^0(u_{1/2}) &=- \int_{\frac{1}{2}}^1 \int_{\fO} 
      ( u_1 - u_0 )\,(\ii \ddb  u_t )^n\,dt. \\
    \end{split}
  \end{equation*}
  Since
  \begin{equation*}
    \begin{split}
        &  \int_0^{\frac{1}{2}} \int_{\fO} ( u_1 - u_0 ) \, (\ii \ddb  u_t )^n
       \,dt - \int_{\frac{1}{2}}^1 \int_{\fO} 
      ( u_1 - u_0 )\, ( \ii \ddb u_t )^n\,dt. \\
        = &  \int_0^{\frac{1}{2}} \int_{\fO} ( u_1 - u_0 ) \bl 
        (\ii \ddb u_t )^n - (\ii \ddb u_{t+1/2} )^n \br \,dt \\
        = & 2 \int_0^{\frac{1}{2}} \int_{\fO} (u_{t+ 1/2} -
        u_t) \bl (\ii \ddb u_t)^n - (\ii \ddb  u_{t+1/2})^n \br \,dt \geq 0 .
    \end{split}
  \end{equation*}
  So
  \begin{equation*}
    F^0(u_1) - F^0(u_{1/2}) \geq  F^0(u_{1/2}) - F^0(u_0).
  \end{equation*}
  The cocycle condition is a simple consequence of the variation
  formula \ref{eq:41}.
\end{proof}

\section{The Convergence}

In this section, let us assume that both $f$ and $\vff$ are independent of $t$.
For $u \in \cP(\fO, \vff)$, define
\begin{equation}
  F(u) = F^0(u) + \int_{\fO} G(z,u) dV,
  \label{eq:45}
\end{equation}
where $dV$ is the volume element in $\CC^n$, and $G(z,s)$ is the
function given by
$$ G(z,s) = \int_0^s e^{-f(z,t)}\,dt. $$
Then the variation of $F$ is 
\begin{equation}
  \fd F(u) = - \int_{\fO} \fd u \bl \det ( u_{\fa\bar{\fb}}) -
  e^{-f(z,u)} \br\, dV.
  \label{eq:46}
\end{equation}

\noindent
\emph{Proof of Theorem~\ref{thm:2}.} We will follow Phong and
Sturm's proof of the convergence of the \KH-Ricci flow in
\cite{PhongSturm2006}. For any $t >0$, the function $u(\cdot, t)$ is in
$\cP(\fO,\vff)$. So by (\ref{eq:46})  
\begin{equation*}
  \begin{split}
    \frac{d\,}{dt} F(u) &= - \int_{\fO} \dot{u} \bl \det(
    u_{\fa\bar{\fb}}) - e^{-f(z,u)} \br \\
    & = - \int_{\fO} \bl \ld(u_{\fa\bar{\fb}}) - (- f(z,u)) \br \bl
    \det(u_{\fa\bar{\fb}}) - e^{-f(z,u)} \br \leq 0.
  \end{split}
\end{equation*}
Thus $F(u(\cdot,t))$ is monotonic decreasing as $t$ approaches
$+\infty$. On the other hand, $u(\cdot,t)$ is uniformly bounded in
$\cC^{2}(\lbar{\fO})$ by (\ref{eq:10}), so both $F^0(u(\cdot,t))$ 
and $f(z,u(\cdot,t))$ are uniformly bounded, hence $F(u)$ is bounded.
Therefore
\begin{equation}
  \int_0^{\infty} \int_{\fO} \bl \ld (u_{\fa\bar{\fb}}) + f(z, u) \br 
   \bl \det(u_{\fa\bar{\fb}}) - e^{-f(z,u)} \br \,dt < \infty.
  \label{eq:47}
\end{equation}
Observed that by the Mean Value Theorem, for $x, y \in \RR$, 
$$ (x+y)(e^x - e^{-y}) = (x+y)^2 e^{\eta} \geq e^{\min(x, -y)} (x-y)^2, $$
where $\eta$ is between $x$ and $-y$. Thus
$$ \bl \ld (u_{\fa\bar{\fb}}) + f \br 
\bl \det(u_{\fa\bar{\fb}}) - e^{-f} \br \geq C_9 \bl 
\ld(u_{\fa\bar{\fb}}) + f\br^2 = C_9 |\dot{u}|^2 $$
where $C_9$ is independent of $t$. Hence
\begin{equation}
  \int_0^{\infty} \llnorm{\dot{u}}^2 \,dt \leq \infty 
  \label{eq:48}
\end{equation}
Let
\begin{equation}
  Y(t) = \int_{\fO} |\dot{u}(\cdot,t)|^2 \, \det (u_{\fa\bar{\fb}}) \,dV, 
  \label{eq:49}
\end{equation}
then
$$  \dot{Y} = \int_{\fO} \bl 2 \ddot{u} \dot{u}  +
    \dot{u}^2 u^{\fa\bar{\fb}} \dot{u}_{\fa\bar{\fb}} \br
    \det(u_{\fa\bar{\fb}})\,dV. $$ 
Differentiate (\ref{eq:1}) in $t$, 
\begin{equation}
 \ddot{u} - u^{\fa\bar{\fb}} \dot{u}_{\fa\bar{\fb}} = f_u \dot{u}, 
 \label{eq:50}
\end{equation}
so
\begin{equation*}
  \begin{split}
  \dot{Y} &= \int_{\fO} \bl 2 \dot{u} \dot{u}_{\fa\bar{\fb}}
   u^{\fa\bar{\fb}} + \dot{u}^2 \big( 2 f_u + \ddot{u} 
   - f_u \dot{u} \big) \br \det(u_{\fa\bar{\fb}})\,dV \\
   &= \int_{\fO} \bl  \dot{u}^2 \big( 2 f_u + \ddot{u} 
   - f_u \dot{u} \big) - 2 \dot{u}_{\fa} \dot{u}_{\bar{\fb}}
   u^{\fa\bar{\fb}} \br \det(u_{\fa\bar{\fb}})\,dV 
   \end{split}
\end{equation*}
From (\ref{eq:50}), we get
$$ \dddot{u} - u^{\fa\bar{\fb}} \ddot{u}_{\fa\bar{\fb}} - f_u
\ddot{u} \leq f_{uu} \dot{u}^2 $$ 
Since $f_u \leq 0 $ and $f_{uu} \leq 0$, so $\ddot{u}$ is bounded
from above by the maximum principle. Therefore
$$ \dot{Y} \leq C_{10} \int_{\fO} \dot{u}^2 \det(u_{\fa\bar{\fb}})\,dV 
= C_{10}Y, $$
and
\begin{equation}
 Y(t) \leq Y(s) e^{C_{10}(t-s)} \qquad \text{for } t > s, 
 \label{eq:51}
\end{equation}
where $C_{10}$ is independent of $t$. By (\ref{eq:48}), 
(\ref{eq:51}) and the uniform boundedness of $\det(u_{\fa\bar{\fb}})$,
we get
$$ \lim_{t\to \infty} \llnorm{u(\cdot, t)} = 0. $$
Since $\fO$ is bounded, the $L^2$ norm controls the $L^1$ norm, hence 
$$ \lim_{t\to \infty} \lnorm{u(\cdot, t)}=0. $$
Notice that by the Mean Value Theorem,
$$ | e^x - 1 | < e^{|x|} |x| $$
so
$$ \int_{\fO} |e^{\dot{u}} - 1 |\,dV \leq e^{\sup|\dot{u}|} \int_{\fO}
|\dot{u}|\,dV $$
Hence $e^{\dot{u}}$ converges to $1$ in $L^1(\fO)$ as $t$ approaches
$+\infty$. Now $u(\cdot,t)$ is bounded in $\cC^2(\lbar{\fO})$, so
$u(\cdot, t)$ converges to a unique function $\tilde{u}$, at least
sequentially in $\cC^1(\lbar{\fO})$, hence $f(z,u) \to f(z,\tilde{u})$
and 
$$ \det(\tilde{u}_{\fa\bar{\fb}}) = \lim_{t\to \infty} 
\det(u(\cdot,t)_{\fa\bar{\fb}}) = \lim_{t\to \infty} e^{\dot{u} -
f(z,u)} = e^{-f(z,\tilde{u})}, $$
i.e. $\tilde{u}$ solves (\ref{eq:8}). 

\qed

\bibliographystyle{alpha}
\bibliography{cpma}

\end{document}